\documentclass[11pt]{article}

\usepackage{amsmath, amssymb, amsthm, verbatim,enumerate,bbm,color} 
\numberwithin{equation}{section}
\usepackage{mathrsfs}

\usepackage[backref,colorlinks,citecolor=blue,bookmarks=true]{hyperref}

\usepackage{tikz}

\title{Weakly Saturated Hypergraphs and a Conjecture of Tuza}

\author{Asaf Shapira \thanks{School of Mathematics, Tel Aviv University, Tel Aviv 69978, Israel. Email: asafico$@$tau.ac.il. Supported in part by ISF Grant 1028/16, ERC Consolidator Grant 863438 and NSF-BSF Grant 20196.} \and
Mykhaylo Tyomkyn
	\thanks{Department of Applied Mathematics, Charles University, Prague, Czech Republic. Email: tyomkyn$@$kam.mff.cuni.cz. Supported in part by GAČR grant 19-04113Y, ERC Synergy Grant DYNASNET 810115 and the H2020-MSCA-RISE Project CoSP-GA No. 823748.}}

\date{\today}
\allowdisplaybreaks

\theoremstyle{plain}
\newtheorem{theorem}{Theorem}[section]
\newtheorem{lemma}[theorem]{Lemma}

\newtheorem{observation}[theorem]{Observation}
\newtheorem{corollary}[theorem]{Corollary}

\makeatletter
\def\moverlay{\mathpalette\mov@rlay}
\def\mov@rlay#1#2{\leavevmode\vtop{%
   \baselineskip\z@skip \lineskiplimit-\maxdimen
   \ialign{\hfil$\m@th#1##$\hfil\cr#2\crcr}}}
\newcommand{\charfusion}[3][\mathord]{
    #1{\ifx#1\mathop\vphantom{#2}\fi
        \mathpalette\mov@rlay{#2\cr#3}
      }
    \ifx#1\mathop\expandafter\displaylimits\fi}
\makeatother

\makeatletter
\renewenvironment{proof}[1][\proofname]
{\par\pushQED{\qed}
	\normalfont\topsep6\p@\@plus6\p@\relax\trivlist
	\item[\hskip\labelsep\bfseries#1\@addpunct{.}]
	\ignorespaces}
{\popQED\endtrivlist\@endpefalse}
\makeatother

\newcommand{\F}{\mathcal F}

\newcommand{\shn}{\mbox{sat}(n,H)}
\newcommand{\whn}{\mbox{wsat}(n,H)}
\newcommand{\whnone}{\mbox{wsat}(n_1,H)}
\newcommand{\whm}{\mbox{wsat}(m,H)}
\newcommand{\whmone}{\mbox{wsat}(m_1,H)}

\newcommand{\eps}{\varepsilon}

\RequirePackage{color}\definecolor{RED}{rgb}{1,0,0}\definecolor{BLUE}{rgb}{0,0,1} 

\begin{document}
\date{}
\maketitle

\begin{abstract}

Given a fixed hypergraph $H$, let $\whn$ denote
the smallest number of edges in an $n$-vertex hypergraph $G$, with the property that
one can sequentially add the edges missing from $G$,
so that whenever an edge is added, a new copy of $H$ is created.
The study of $\whn$ was introduced by Bollob\'as in 1968,
and turned out to be one of the most influential topics in extremal combinatorics.
While for most $H$ very little is known regarding $\whn$,
Alon proved in 1985 that for every graph $H$ there is a limiting constant $C_H$ so that
$\whn=(C_H+o(1))n$. Tuza conjectured in 1992 that Alon’s theorem can be (appropriately) extended to arbitrary $r$-uniform hypergraphs. In this paper we prove this conjecture.

\end{abstract}

\section{Introduction}\label{sec:intro}

Typical problems in extremal combinatorics ask how large or small a discrete structure can be, assuming it possesses certain properties.
For example, the Tur\'an problem asks, for a fixed $r$-uniform hypergraph ($r$-graph for short) $H$, to determine the smallest integer
$m=\mbox{ex}(n,H)$ so that every $n$-vertex $r$-graph with $m+1$ edges has a copy of $H$.
Another example is the Ramsey problem which asks to find the minimum
integer $R=R(n)$ so that every $2$-coloring of the edges of the complete graph on $R$ vertices has a monochromatic clique of size $n$.
While in many cases it seems hopeless to obtain full solutions to these problems, one would at least like to know that
these extremal functions are ``well behaved''. For example, it is natural to ask if the quantities $\mbox{ex}(n,H)/n^r$ and $R(n)^{1/n}$ tend to a limit. While it is easy to see that the first quantity indeed tends to a limit~\cite{KNS64}, it is a famous open problem of Erd\H{o}s~\cite{erdos1981problems, erdos1981some, erdHos1981combinatorial} to prove that the second one does so as well. Our aim in this paper is to prove that another well studied extremal function is well behaved.

For a set of vertices $V$ we use $\binom{V}{r}$ to denote the complete $r$-graph on $V$.
For a fixed $r$-graph $H$, an $r$-graph $G=(V,E)$ is called \emph{$H$-saturated} if it does not contain a copy of $H$ but for any edge $e \in \binom{V}{r} \setminus E(G)$
adding $e$ to $G$ creates a copy of $H$. We let $\shn$ denote the smallest number of edges in an $H$-saturated
$r$-graph on $n$ vertices. Let $K^{r}_t$ denote the complete $r$-graph on $t$ vertices; when $r=2$ (i.e. when dealing with graphs) we use
$K_t$ instead of $K^{2}_t$. The problem of determining $\mbox{sat}(n,{K_t})$
was raised by Zykov~\cite{Zy49} in the 1940’s and studied in the 1960's by Erd\H{o}s, Hajnal and Moon~\cite{EHM64} who
showed that $\mbox{sat}(n,K_t)=\binom{n}{2}-\binom{n-t+2}{2}$. Their result was later generalized by Bollob\'as~\cite{BB65} who showed that $\mbox{sat}(n,K^{r}_t)=\binom{n}{r}-\binom{n-t+r}{r}$. It is worth noting that the proof in~\cite{BB65} introduced the (equivalent and) highly influential \emph{Two Families Theorem}, stating that if $A_1,\ldots,A_s$ and $B_1,\ldots,B_s$ are two families of sets, so that all $|A_i|=a$, all $|B_i|=b$, and $A_i \cap B_j =\emptyset$ if and only if $i= j$, then $s \leq \binom{a+b}{a}$.

We say that $G$ is \emph{weakly $H$-saturated}
if the edges of $\binom{V}{r} \setminus E(G)$ admit an ordering $e_1,\dots,e_k$ such that for each $i=1,\dots,k$ the $r$-graph $G_i:=G\cup \{e_1,\dots,e_i\}$ contains a copy of $H$ containing the edge $e_i$. We refer to the sequence $e_1,\dots,e_k$ as a \emph{saturation process}.
Define $\whn$ to be the smallest number of edges in a weakly $H$-saturated
$r$-graph on $n$ vertices. Note that we may automatically assume that any $G$ realizing $\whn$ is $H$-free, as otherwise we could remove an edge from a copy of $H$ in $G$ to obtain a smaller weakly $H$-saturated $r$-graph. Hence weak saturation can be viewed as an extension of the notion of (ordinary) saturation.

The problem of determining $\whn$ was first introduced in 1968 by Bollob\'as~\cite{BB68} who conjectured
that $\mbox{wsat}(n,{K_t})=\mbox{sat}(n,{K_t})$. This was proved independently by Frankl~\cite{Frankl82} and Kalai~\cite{Kal84, Kal85} using the skewed\footnote{In the skewed version one assumes that $A_i \cap B_i = \emptyset$ as in Bollob\'as's theorem, but that $A_i \cap B_j \neq  \emptyset$ only for $i < j$.} variant of Bollob\'as's Two Families Theorem (a related statement for matroids was proven earlier by Lov\'asz~\cite{Lovasz77}) and further extended by Alon~\cite{Alon} and Blokhuis~\cite{blokhuis1990solution}. This result, which has several other equivalent formulations, is amongst the most classical and important results of extremal combinatorics. See e.g. the discussions in \cite{BFbook88, ov18, pin1983two, scottwilmer19}.

While the aforementioned results determine the exact value of $\whn$ when $H=K^{r}_t$, our understanding
of this function for general $H$ is much more limited, despite decades of extensive study~\cite{Alon, BBMR, Balogh98, btt21,
Erdos91, FaudreeGould14,  Morrison18, MoshkShap, Pikhurko01a,  Pikhurko01b, Semanivsin97, Sidorowicz07, Tuza88, Tuza92}. Note that by the construction from~\cite{EHM64}, we know that every graph $H$ we have
\begin{equation}\label{eqSimple}
\whn \leq \shn \leq \mbox{sat}(n,{K_{|V(H)|}})=O_H(n).
\end{equation}
As of now, the best known general bounds for $\whn$ when $H$ is a graph are due to Faudree, Gould and Jacobson~\cite{Faudree13} who showed that for graphs $H$ of minimum degree $\delta=\delta(H)$ we have\footnote{The upper bound is known to be tight for many graphs, the cliques being one example. Concerning the lower bound, the authors of~\cite{Faudree13} give a construction of a graph $H$ with $\whn \leq (\delta/2+1/2-1/\delta)n$.}
$$
\left(\frac{\delta}{2}-\frac{1}{\delta+1}\right)\cdot n \leq \whn \leq (\delta-1)\cdot n+O(1).
$$
\noindent
At this point it is natural to ask if for every $H$ there is a constant $C_H$ so
that
\begin{equation}\label{eqNoga}
\whn =(C_H+o(1))n.
\end{equation}
Such a result was obtained in 1985 by Alon~\cite{Alon}, who proved
that for graphs the function $\whn$ is (essentially) subadditive, implying that $\whn/n$ tends to a limit, by Fekete's subadditivity lemma \cite{fekete}.

Much less was known when $H$ is an $r$-graph with $r \geq 3$.
Similarly to the case $r=2$ above (\ref{eqSimple}), Bollob\'as's construction from~\cite{BB65} gives a simple bound of
$$
\whn\leq \shn =O_H(n^{r-1}).
$$
A more refined result was obtained by Tuza~\cite{Tuza92} who introduced
the following key definition. The {\em sparseness} of an $r$-graph $H$, denoted $s(H)$, is the smallest size of a vertex set $W\subseteq V$ contained in precisely one edge of $H$; note that $1\leq s(H)\leq r$ for every non-empty $r$-graph $H$.
It was proved in \cite{Tuza92} that for every $r$-graph $H$ there are two positive reals $c_H$ and $C_H$ such that
\begin{equation}\label{eqTuza}
c_H\cdot n^{s-1} \leq \whn \leq C_H \cdot n^{s-1}.
\end{equation}
It was further conjectured in \cite{Tuza92} that
the more refined bound
$\whn=C_H\cdot n^{s-1} +O(n^{s-2})$ holds for every $r$-graph of sparseness $s$. See also the recent survey \cite{2021survey} on saturation problems where this conjecture is further discussed.
Since such a result is not known even for graphs (i.e.~when $r=s=2$), Tuza~\cite{Tuza92} asked if one can
improve upon (\ref{eqTuza}) by showing that for every $r$-graph we have $\whn=C_H\cdot n^{s-1} +o(n^{s-1})$ where $s=s(H)$.
Prior to this work, such a result was only known for $r=2$ by Alon's result (\ref{eqNoga}).
In this paper we fully resolve Tuza's problem for all $r$-graphs.
\begin{theorem}\label{thm:main}
	For every $r$-graph $H$ there is $C_H>0$ such that
	$$
	\lim_{n\rightarrow \infty}{\whn}/{n^{s-1}}=C_H,
	$$
where $s=s(H)$ is the sparseness of $H$. In particular\footnote{Here we simply use the fact that for every $r$-graph $H$ we have $1 \leq s(H)\leq r$.}, for every $r$-graph $H$ there is $C'_H\geq 0$ such that
	$$	
\lim_{n\rightarrow \infty}{\whn}/{n^{r-1}}=C'_H.
	$$
\end{theorem}
It is interesting to note that Tuza~\cite{Tuza86} (for graphs) and Pikhurko~\cite{pikhurko_1999} (for arbitrary $r$-graphs) also conjectured that a theorem analogous to the second assertion of Theorem \ref{thm:main} should hold with respect to $\shn$. However, there are results suggesting that this analogous statement does not hold even for graphs,
see~\cite{behague2018, CL20, Pikhurko04} and the discussion in \cite{2021survey}.

\paragraph{Proof and paper overview:}

It is natural to ask why Alon's~\cite{Alon} one-paragraph proof of Theorem \ref{thm:main} for $s=2$ is hard
to extend to $s > 2$.\footnote{While formally~\cite{Alon} only deals with $r=2$, the proof very similarly applies to $s=2$ for arbitrary $r$.} Perhaps the simplest reason is that one cannot hope to show that in these cases the
function $\whn$ is subadditive (and then apply Fekete's lemma) since a subadditive function is necessarily
of order $O(n)$, while we know from (\ref{eqTuza}) that when $s \geq 3$ the function $\whn$ is of order at least $n^2$.
One can of course try to come up with more complicated recursive relations for $\whn$ and combine them with variants of Fekete's lemma,
but this seems to lead to a dead-end (we have certainly tried to go down that road).
The main novelty in this paper is in finding a direct and efficient way to use an $m$-vertex $r$-graph witnessing the fact that
$\whm$ is small, in order to build arbitrarily large $n$-vertex $r$-graphs witnessing the fact that $\whn$ is small.
One of the main tools we use to construct such an example is R\"odl's approximate designs theorem \cite{Rodl} which enables
us to efficiently combine many examples of size $m$ into one of size $n$. R\"odl's result
would only allow us to construct a saturation process generating part of the edges of $K_n^r$.
To complete this saturation process we would also need another set of gadgets.
In Section~\ref{sec:prelim} we establish some general facts about weak saturation of $r$-graphs. The main proof of Theorem~\ref{thm:main} is carried out in Section~\ref{sec:proof}.

\section{Preliminaries}\label{sec:prelim}

In this section we establish a few useful facts regarding $\whn$.
Perhaps counterintuitively, a graph $G$ can be weakly $H$-saturated but not weakly $H'$-saturated for some subgraph $H'\subseteq H$.
In fact, $\whn$ is not even monotone with respect to $H$.
For example, if $H'$ is a triangle and $H$ is a triangle with a pendant edge, then $\mbox{wsat}(n,{H'})=n-1$ (with extremal examples being all $n$-vertex trees), while $\whn=3$ (the triangle being one extremal example). We now define a setting where one does have such a monotonicity.

Given $s\leq r\leq h$, let $T^-_{r,h,s}$ be the $r$-graph obtained from the complete $h$-vertex $r$-graph $K^{r}_h$ by choosing a set $Z$ of $s$ vertices and deleting all edges containing $Z$ as a subset.
Define the \emph{template $r$-graph} $T_{r,h,s}$ to be the (unique up to isomorphism) $r$-graph obtained from $T^-_{r,h,s}$ by adding a single missing edge $f$ (on the same vertex set), we call $f$ the \emph{special edge}. To practise the definition, note that $T_{r,h,r}$ is simply the clique $K_h^r$. We say that an $r$-graph $G$ is \emph{$T_{r,h,s}$-template saturated} if the edges in $\binom{V(G)}{r}\setminus E(G)$ admit an ordering  $e_1,\dots, e_k$ (the \emph{$T_{r,h,s}$-template saturation process}) such that for each $i=1,\dots,k$ the $r$-graph $G_i:=G\cup \{e_1,\dots,e_i\}$ contains a copy of $T_{r,h,s}$ in which the edge $e_i$ plays the role of the special edge $f$.
The next lemma shows that comparing $T_{r,h,s}$-template saturation with weak $H$-saturation, for an $r$-graph $H$ with $s(H)=s$, we do have monotonicity.

\begin{lemma}\label{lem:template}
Suppose $G$ and $H$ are $r$-graphs with $|V(H)|=h$ and $s(H)=s\geq 2$. Suppose that $G$ is $T_{r,h,s}$-template saturated. Then $G$ is weakly $H$-saturated.
\end{lemma}
\begin{proof}
By the definition of sparseness, $H$ contains a set $S$ of $s$ vertices contained in precisely one edge $e\in E(H)$. Deleting $e$ from $H$ gives the $r$-graph $H^-$ of order $h$ and in which no edge contains $S$ as a subset. By the definition of $T_{r,h,s}$ we have that $H^-$ is a subgraph of $T^-_{r,h,s}$. More importantly, $H^-$ can be embedded into $T^-_{r,h,s}$ in a way that maps $S$ bijectively on $Z$. 
Indeed, any map $\phi:V(H^-)\mapsto V(T^-_{r,h,s})$ which sends the set $S$ of $H^-$ to the set $Z$ of $T^-_{r,h,s}$ has this property.

Consider now a $T_{r,h,s}$-template saturation process of $G$. By the above argument, at every step the newly created copy of $T_{r,h,s}$ (with the new edge playing the role of the special edge) gives rise to a new copy of $H$, where the new edge plays the role of $e$. Therefore, the same process certifies weak $H$-saturation of $G$.
\end{proof}
\noindent

We will frequently use the following simple observation stating that saturation processes are monotone with respect to the starting graph $G$.

\begin{observation}\label{obs:mono}
	For any $r$-graphs $G$ and $H$ with $|V(G)|=n$, if $G$ is weakly $H$-saturated then so is any intermediate $r$-graph $G\subseteq G'\subseteq K_n^r$. The analogous statement holds for $T_{r,h,s}$-template saturation.
\end{observation}
\noindent
As an immediate consequence we obtain

\begin{lemma}\label{lem:trhslarge}
	Suppose $s'$ satisfies $r\geq s'\geq s\geq 2$, and let $G$ be a supergraph of $T^-_{r,h,s'}$ on the same vertex set. Then $G$ is $T_{r,h,s}$-template saturated in $K^{r}_{h}$.
\end{lemma}
\begin{proof}
The assertion is true for $G=T_{r,h,s}$: the missing edges can be added in any order. For arbitrary $s'\geq s$, the $r$-graph $T^-_{r,h,s'}$ and, by extension, every supergraph thereof, contain $T^-_{r,h,s}$ as a subgraph. Therefore, the assertion holds by Observation~\ref{obs:mono}.
\end{proof}

Our next goal is to obtain a certain ``approximate continuity'' of $\whn$ with respect to $n$.
We first need the following lemma.
\begin{lemma}\label{lem:annoying}
Let $h\geq r\geq s\geq 2$, suppose $V=A\sqcup B$ is a set of vertices, where $|B| \leq |A|$, and let $E=\binom{A}{r}$ be the edges contained in $A$.
Then there exists a set $E'\subseteq \binom{V}{r}$ of size at most $rh^r|A|^{s-2}|B|$ so that $G=(V,E\cup E')$ is $T_{r,h,s}$-template saturated in $\binom{V}{r}$.	
\end{lemma}	

\begin{proof}
Let $C\subseteq A$ be a fixed set of $h$ vertices, and let
$$E':= \{f\in \binom{V}{r}\setminus E\colon |f\setminus C|\leq s-1\}.$$
Note that every such $f$ contains at least one vertex from $B$ (as otherwise we would have $f\in E$).
Since $|B| \leq |A|$ we have $|E'|\leq rh^r|A|^{s-2}|B|$.
We claim that $G=(V,E\cup E')$ is $T_{r,h,s}$-template saturated, as desired. To describe the corresponding saturation process, we consider a missing edge $f$ and apply induction on $\lambda(f):=|f\setminus C|$. The base case of $\lambda(f)\leq s-1$ is given by the fact that these edges are already in $E\cup E'$.

Suppose now that $\lambda\geq s$ is arbitrary, $f$ is a missing edge with $\lambda(f)=\lambda$, and every edge $e$ with $\lambda(e)<\lambda$ has already been added. Let $L:=f\setminus C$ (so that $|L|=\lambda$), and let $P\subseteq C\setminus f$ be a set of $h-r$ vertices.
By the induction hypothesis, all edges on the vertex set $P\cup f$ not containing $L$ as a subset have already been added. Conversely, every currently missing edge must contain $L$ as a subset, which means the currently present edges on $P\cup f$ form a supergraph of $T^-_{r,h,\lambda}$. Since $\lambda\geq s$, by Lemma~\ref{lem:trhslarge} we can add all missing edges on the set $P\cup f$, including the edge $f$, via a $T_{r,h,s}$-template saturation process. This completes the induction step.
\end{proof}

In the following statement the reader should think of $k_2 = o(k_1)$. Since $\mbox{wsat}(k_1,H)$ is of order
$k^{s-1}_1$ (by (\ref{eqTuza})) this means that in this regime $\mbox{wsat}(k_1+k_2,H)=(1+o(1))\mbox{wsat}(k_1,H)$.

\begin{corollary}\label{cor:monotone}
Let $h\geq r\geq s\geq 2$ and $H$ be an $r$-graph with $|V(H)|=h$ and $s(H)=s$.
Then for every $k_2 \leq k_1$ we have
$$
\mbox{wsat}(k_1+k_2,H) \leq \mbox{wsat}(k_1,H) + rh^r\cdot k_1^{s-2}\cdot k_2.
$$

\begin{proof}
Given a minimal weakly $H$-saturated $r$-graph $G^-=(A,E^-)$ on $k_1$ vertices, construct a weakly $H$-saturated $r$-graph $G=(V,E)$ on $k_1+k_2$ vertices as follows. Let $B$ be a set of $k_2$ vertices disjoint from $A$, let $V:=A\sqcup B$ and $E:=E^-\cup E'$ where $E'$ is the edge set as described in Lemma~\ref{lem:annoying}. Then $G$ is weakly $H$-saturated. Indeed, first run a saturation process inside $A$. Afterwards the remaining missing edges can be added by Lemma~\ref{lem:annoying} and Lemma~\ref{lem:template}. Moreover, by Lemma~\ref{lem:annoying} we have
$$
|E|\leq |E^-|+|E'|\leq \mbox{wsat}(k_1,H) + rh^r\cdot k_1^{s-2}\cdot k_2.
$$
\end{proof}
\end{corollary}

\section{Proof of Theorem~\ref{thm:main}}\label{sec:proof}

As we mentioned at the end of Section~\ref{sec:intro}, our approach to proving Theorem~\ref{thm:main} is to use an $m$-vertex weakly $H$-saturated graph with few edges in order to build, for all large enough $n$, an $n$-vertex weakly $H$-saturated graph with few edges. In the first step of the proof we will take $\ell$ disjoint vertex ``clusters'' (for some large $\ell$) and cover them with copies of the $m$-vertex example. To do so efficiently, we shall need the following classical theorem of R\"odl~\cite{Rodl} (formerly, the Erd\H{o}s-Hanani conjecture).
\begin{theorem}[R\"odl \cite{Rodl}]\label{prop:Rodl}
	For every $k>t>1$ and $\delta>0$ for all $N>N_0(k,t,\delta)$ the following holds. There exists a collection $\F\subseteq \binom{[N]}{k}$ of size at most $(1+\delta)\binom{N}{t}/\binom{k}{t}$ such that every $A\in \binom{[N]}{t}$ is contained in some $F_A\in \F$.
\end{theorem}
\noindent
The outcome of applying R\"odl's theorem will be a graph (denoted $G'_n$ in the proof of Theorem \ref{thm:main}) that has an $H$-saturation process generating part of the edges of $K^r_n$, namely the edges containing vertices from at most $s-1$ of the $\ell$ clusters. To generate the remaining edges, we will add to $G'_n$ another collection of gadgets (the edge set $E_2$ in the proof of Theorem \ref{thm:main}). These are described in the next two lemmas. We note that the bound guaranteed by Lemma \ref{lem:percolate} is crucial for establishing
that $|E_2|=o(n^{s-1})$, thus making sure that these extra edges have a negligible effect on the total number of edges of the graphs
we construct.

\begin{lemma}\label{lem:spartite}
	Suppose $G=(V,E)$ is an $r$-graph such that $V=\bigsqcup_{i=1}^s V_i$ with $|V_i|\geq h$ for all $i$
	and $E$ contains all $r$-tuples in $V$ missing at least one of the sets $V_i$. For each $i\in [s]$ let $R_i\subseteq V_i$ be a set of $h$ vertices. Let $E'$ be the set of all edges containing at least $r-s+2$ vertices from $R:=\bigcup_{i}R_i$.
	Then $E\cup E'$ is $T_{r,h,s}$-template saturated in $\binom{V}{r}$.
\end{lemma}
\begin{proof}
	For each $i\in [s]$ let $L_i:=V_i\setminus R_i$ and let $L:=\bigcup_i L_i$. Let the vertices of $R$ and $L$ be called \emph{rigid} and \emph{loose}, respectively. Our aim is to define a  $T_{r,h,s}$-template saturation process. Note that by assumption the edges in $\binom{V}{r}$ containing at most $s-2$ loose vertices are already present.
	
	Consider first the missing edges $C\in \binom{V}{r}\setminus (E\cup E')$ containing exactly $s-1$ loose vertices.
	By pigeonhole, for every such edge there is an index $j\in [s]$ such that no vertex in $C_j:=C\cap V_j$ is loose. Let
	$$\rho(C):=\min\{|C_j|\colon C_j\subseteq R\}.$$
	
	We now apply induction on $\rho$ in order to construct a $T_{r,h,s}$-template saturation process adding successively the edges with $\rho=0,1,2,\dots$. For the base case $\rho=0$, note that such edges necessary do not contain
any vertex from (at least) one of the sets $V_1,\ldots,V_s$, and therefore are already in $E$.
	
	For the induction step let $\rho(C)\geq 1$ be arbitrary, and suppose that the edges with a smaller value of $\rho$ are already present. Let $j\in [s]$ satisfy $C_j\subseteq R$ and $|C_j|=\rho$,
	let $i\in [s]\setminus \{j\}$ be another index and let $D\subseteq R_i\setminus C_i$ be a set of size $h-r$. Observe now that inside the set $D\cup C$ the only edges not yet present are the ones containing $(C\cap L)\cup C_j$ as a subset. Indeed,
since $(D\cup C)\cap L = C\cap L$ every edge in $D \cup C$ missing a vertex from $C\cap L$, contains at most $s-2$ loose vertices, and is thus in $E'$. Furthermore, every edge in $D \cup C$ missing a vertex from $C_j$ contains fewer than $\rho(C)$ from $R_j$ (and no vertex from $L_j$). Therefore, it is already present by the induction hypothesis. Thus the currently present edges on $D \cup C$ induce a supergraph of $T^-_{r,h,s'}$, where $s'=|(C\cap L)\cup C_j|$. Since $s'=s-1+\rho(C) \geq s$, by Lemma~\ref{lem:trhslarge} we can add all the missing edges on $D\cup C$, including $C$, via a $T_{r,h,s}$-template saturation process.

Now consider the missing edges $C$ having at least $s$ loose vertices and apply induction on $\lambda(C):=|C\cap L|$; we can view the case $\lambda(C)=s-1$ treated above as the base case. For the induction step, suppose that $\lambda(C)\geq s$ is arbitrary and that all the edges with a smaller value of $\lambda$ are already present. Let $D\subseteq R\setminus C$ be an arbitrary set of $h-r$ vertices. Then, by the induction hypothesis, all edges on $D\cup C$ not already present contain $C\cap L$
as a subset (for otherwise they would have fewer than $\lambda(C)$ loose vertices). Hence, the currently present edges within $D\cup C$ induce a supergraph of $T^-_{r,h,\lambda(C)}$. Since $|C\cap L|=\lambda(C)\geq s$, by Lemma~\ref{lem:trhslarge} we can add all of the missing edges on $D\cup C$, including $C$, applying a $T_{r,h,s}$-template saturation process.
	
Having reached $\lambda=r$, we have covered all edges in $\binom{V}{r}$.
\end{proof}

\begin{lemma}\label{lem:percolate}
	Suppose $V=\bigsqcup_{i=1}^{\ell} V_i$ for some $\ell\geq s$, with $V_i\geq h$ for all $i$. Suppose further that
	for each $i\in [\ell]$ there is a designated subset $R_i\subseteq V_i$ with $|R_i|=h$. Let $G=(V,E)$ be an $r$-graph with
	$E=E_1\cup E_2$ where $E_1$ contains all edges hitting at most $s-1$ different $V_i$ and
	$$E_2:=\bigcup_{Q\in \binom{[\ell-1]}{s-1}}E_2(Q),$$
	where $E_2(Q)$ is a copy of $E'$ as in Lemma~\ref{lem:spartite} on $V_Q:=V_\ell\sqcup\bigsqcup_{q\in Q}V_q$. Then $G$ is $T_{r,h,s}$-template saturated in $\binom{V}{r}$. Moreover, if $|V_i|=t$ for all $i$, then we have
	$$|E_2|\leq rh^{r-s+2}\binom{\ell-1}{s-1}t^{s-2}.$$
\end{lemma}
\begin{proof}
	First, for each $Q\in \binom{[\ell-1]}{s-1}$ consider the induced subgraph $G[V_Q]$. Note that with the partition $V_Q=V_\ell \sqcup \bigsqcup_{q\in Q}V_q$ this $r$-graph contains all the edges in the statement of Lemma~\ref{lem:spartite}. Hence, by Observation~\ref{obs:mono} and Lemma~\ref{lem:spartite} we can apply a $T_{r,h,s}$-template saturation process in order to add all missing edges inside $V_Q$. Thus we may assume from here on that the edges inside all sets $V_Q$ are present.
	
	For an edge $e\in \binom{V}{r}$ let $J(e)=e\setminus V_\ell$ and $j(e)=|J(e)|$. By the above, every edge $e$ with $j(e)\leq s-1$ has already been added and, conversely, every missing edge $e\in \binom{V}{r}$ satisfies $j(e)\geq s$. We construct a $T_{r,h,s}$-template saturation process for the missing edges by adding them successively: first the edges with $j(e)=s$, followed by $j(e)=s+1,\dots,j(e)=r$.
	To do so we apply induction on $j=j(e)$, where $j\leq s-1$ can be viewed as the base case.
	
	For the induction step, fix $j$ and suppose that all edges $e'$ with $j(e')<j$ have already been added. Let $e$ be an arbitrary edge with $J(e)=:J$ and $j(e)=j$, and consider the set $T=e\cup P$ where $P\subseteq V_\ell\setminus e$ is an arbitrary set of $h-r$ vertices disjoint from $e$; clearly, we have $|T|=h$. Notice now that every potential edge $f \subseteq T$ satisfies either $f\supseteq J$ or $|f\cap J|<j$. In the latter case, $j(f)<j$, so by the induction hypothesis, $f$ has already been added.
Thus, the only edges missing from $T$ are the ones containing $J$ as a subset. In other words, the edges currently present induce on $T$ a supergraph of $T^-_{r,h,j}$. However, since $j\geq s$, by Lemma~\ref{lem:trhslarge} we can add all the remaining edges of $\binom{T}{r}$, including $e$, via a $T_{r,h,s}$-template saturation process. Since $e$ was arbitrary subject to $j(e)=j$, this proves the induction step.
	
	For the last assertion of the lemma, simply notice that, by construction in Lemma~\ref{lem:spartite}, each $E_2(Q)$ is of size at most $rh^{r-s+2}t^{s-2}$.
\end{proof}
\begin{proof}[Proof of Theorem~\ref{thm:main}]
Let $H$ be an $r$-graph with $|V(H)|=h$. Suppose first that $s(H)=1$, and observe that in this case
$\whn \leq \binom{h}{r}$ holds for every $n \geq h$. Indeed, take a set of $n$ vertices and put a copy
of $K^r_h$ on $h$ of the vertices. Pick any other vertex $v$ not in the copy of $K^r_h$, and note that since $s(H)=1$
adding an edge containing $v$ and $r-1$ of the vertices of $K^r_h$ is guaranteed to form a copy of $H$. Hence
there is an $H$-saturation process that starts with the initial $K^r_h$ and ends with $K^r_{h+1}$. We can then
turn the $K^r_{h+1}$ into $K^r_{h+2}$ etc, until we obtain a complete $r$-graph on the $n$ vertices.
We can thus define $C_H:=\min\{\whn:n\geq h\}$, and let $n_1\geq h$ satisfy $\whnone=C_H$.
By the same reasoning as above, we also have $\whn \leq \whnone$ for every $n\geq n_1$ (we first
obtain $K^r_{n_1}$ and them complete it to $K^r_n$). By minimality of $C_H$ we must have $\whn= \whnone$. Therefore, $\lim_{n\rightarrow \infty}\whn/n^{s-1}=C_H$.
	
Hence, from now on let us assume that $s(H)=s\geq 2$. Let
$$
C_H:=\liminf_{n\rightarrow \infty} \whn/n^{s-1}.
$$
For brevity we shall write $C$ for $C_H$. Recall that
by Tuza's theorem (\ref{eqTuza}), we know that for every large enough $n$ we have $c_1 n^{s-1}\leq \whn\leq c_2 n^{s-1}$ for some positive constants $c_2(H)\geq c_1(H)>0$,
implying that $C>0$. We now claim that $C$ satisfies the assertion of Theorem~\ref{thm:main}. To this end we prove that for every $\eps>0$ we have $\whn \leq (C+8\eps)n^{s-1}$ for all large enough $n$.

Let $\eps>0$ satisfy $\eps<\eps_0(H)$ where $\eps_0$ is chosen so as to satisfy the inequalities required in the proof, and let $m_1$ satisfy $(i)$ $\whmone\leq (C+\eps)m_1^{s-1}$ and $(ii)$ $m_1 \geq m_0(\eps,H)$
so as to satisfy the various inequalities we require in the proof below. Note that by our choice of $C$
there are infinitely many values of $m_1$ satisfying condition $(i)$ hence we can always find $m_1$ satisfying
condition $(ii)$ as well.
Let $m=\lceil m_1^{1/(s-1)}\rceil^{s-1}$ be the next largest perfect $(s-1)$-st power.
Since
$$
m=m_1+O(m_1^{(s-2)/(s-1)}),
$$
we can deduce from Corollary~\ref{cor:monotone} (with $k_1=m_1$ and $k_2=m-m_1$) that
\begin{equation}\label{eqm}
\whm\leq \whmone+O(m_1^{s-2}m_1^{(s-2)/(s-1)})\leq  (C+\eps)m_1^{s-1} +\eps m_1^{s-1}=
(C+2\eps)m^{s-1},
\end{equation}
where the second inequality uses the fact that $m_1 \geq m_0(\eps,H)$.
We now claim that for all sufficiently large $n \geq n_0(m_1,\eps,h)$ we have $\whn\leq (C+8\eps)n^{s-1}$.
To this end, it suffices to show that for every large enough $n$ which is a multiple of $m^{1/(s-1)}$, we have
\begin{equation}\label{eqfinal}
\whn\leq (C+7\eps)n^{s-1}\;.
\end{equation}
Indeed, assuming this, let $n$ be arbitrary and set $n_1=m^{1/(s-1)} \cdot \lfloor n/m^{1/(s-1)} \rfloor$.
By Corollary~\ref{cor:monotone} (with $k_1=n_1$ and $k_2=n-n_1=O(m^{1/(s-1)})$) and (\ref{eqfinal}) we would get that
$$
\whn \leq \whnone+O(n_1^{s-2}m^{1/(s-1)})\leq (C+7\eps)n_1^{s-1}+\eps n_1^{s-1}= (C+8\eps)n^{s-1},
$$
where the second inequality uses the fact that $n \geq n_0(m,\eps,h)$.

To prove (\ref{eqfinal}) let $m$ and $n$ be as above, let $V'$ be a set of $n/m^{1/(s-1)}$ vertices and let $V$ be a set of $n$ vertices, obtained by replacing each $v\in V'$ by a cluster $S_v$ of $m^{1/(s-1)}$ vertices.

For all large enough $n\geq n_0(m,\eps,h)$ by R\"odl's theorem (Theorem~\ref{prop:Rodl}, applied with $N=n/m^{1/(s-1)}$, $k=m^{1-1/(s-1)}$, $t=s-1$ and $\delta=\eps/C$) there is a collection $\mathcal{D}$ of at most
$$
(1+\delta)\frac{\binom{n/m^{1/(s-1)}}{s-1}}{\binom{m^{1-1/(s-1)}}{s-1}}\leq (1+3\delta)\frac{n^{s-1}}{m^{s-1}}
$$
subsets of $V'$ of size $m^{1-1/(s-1)}$, so that each $(s-1)$-tuple of vertices $\{v_1,\dots,v_{s-1}\}\subseteq V'$ belongs to at least one $D\in \mathcal{D}$. The inequality holds assuming\footnote{Indeed, denote $p=m^{1-1/(s-1)}$. If $m$ is large enough so that $p-s \geq (1-\frac{\delta}{2s})p$, then $\binom{n/m^{1/(s-1)}}{s-1}/\binom{p}{s-1} \leq (n^{s-1}/m)/\prod^{s-2}_{i=0}(p-i) \leq (n^{s-1}/m)/(1-\delta/2s)^{s-1}p^{s-1} \leq (1+\delta)n^{s-1}/m^{s-1}$.} $m\geq m_0(\eps,H)$.

Define an $r$-graph $G'_n$ as follows: go over all $D\in \mathcal{D}$ one by one in any order and apply the following procedure. Suppose $D=\{v_1\dots,v_{t}\}$, where $t=m^{1-1/(s-1)}$ and let $S_D=S_{v_1}\cup\dots\cup S_{v_t}$ be the corresponding $m$ vertices in $V$. By (\ref{eqm}) there is a weakly saturated $r$-graph on $m$ vertices with at most $(C+2\eps)m^{s-1}$ edges, denoted $G_m$; put a copy of $G_m$ on $S_D$. Let $G'_n$ be the union over all $S_D$. Then, since $\delta=\eps/C$ and assuming $\eps < \eps_0(H)$ we have 
\begin{equation}\label{eqG}
|E(G'_n)|\leq |\mathcal{D}||E(G_m)|\leq (1+3\delta)\frac{n^{s-1}}{m^{s-1}}(C+2\eps)m^{s-1}\leq (C+6\eps)n^{s-1}.
\end{equation}

To complete the definition of $G_n$, we take $E(G_n)=E(G'_{n})\cup E_2$, where $E_2$ is as in Lemma~\ref{lem:percolate}, with
the parameters $\ell=n/m^{1/(s-1)}$, $t=m^{1/(s-1)}$ and the clusters $\{S_v:v\in V'\}$ playing the role of $V_1,\dots V_\ell$.
By Lemma~\ref{lem:percolate} we have
$$
|E_2|\leq rh^{r-s+2}\binom{\ell-1}{s-1}t^{s-2}=rh^{r-s+2}\binom{\frac{n}{m^{1/(s-1)}}-1}{s-1}m^{\frac{s-2}{s-1}}=O\left(\frac{n^{s-1}}{m^{1/(s-1)}}\right)\leq \eps n^{s-1},
$$
where the last inequality assumes $m \geq m_0(\eps,H)$.
Combining this with (\ref{eqG}) we have
$$
|E(G_n)|\leq (C+7\eps)n^{s-1}.
$$
Hence, to complete the proof of (\ref{eqfinal}),
it remains to describe an $H$-saturation process for $G_n$. Note by definition of $G'_n$,
for each $D \in {\cal D}$ there is an $H$-saturation process for completing all hyperedges in $S_D$ (namely,
the $H$-saturation process of $G_m$, or of a supergraph of it). Since the sets in ${\cal D}$ cover all $(s-1)$-tuples $\{u_1,\dots,u_{s-1}\} \subseteq V'$, once all these processes are complete, we have all hyperedges
$\{v_1,\dots,v_r\} \subseteq V$, hitting at most $s-1$ different sets $S_{u}$. Then, by Observation~\ref{obs:mono} and Lemma~\ref{lem:percolate}, our $r$-graph $G_n$ is $T_{r,h,s}$-template saturated, which by Lemma~\ref{lem:template} implies it is weakly $H$-saturated. This completes the $H$-saturation process of $G_n$ in $K^{r}_n$.
\end{proof}

\bibliography{wsat_hyp}

\end{document}